\documentclass[11pt]{article}
\usepackage[margin=1.1in]{geometry}
\usepackage{amsmath,amssymb,amsthm,mathtools}
\usepackage{booktabs}
\usepackage{float}
\usepackage[colorlinks=true,linkcolor=blue,citecolor=blue]{hyperref}

\theoremstyle{plain}
\newtheorem{theorem}{Theorem}
\newtheorem{lemma}[theorem]{Lemma}
\newtheorem{corollary}[theorem]{Corollary}
\newtheorem{proposition}[theorem]{Proposition}
\theoremstyle{remark}
\newtheorem{remark}[theorem]{Remark}
\theoremstyle{plain}
\newtheorem{conjecture}[theorem]{Conjecture}
\theoremstyle{definition}

\newcommand{\Li}{\operatorname{Li}}
\newcommand{\Cl}{\operatorname{Cl}}
\newcommand{\FT}{\mathcal{A}}
\newcommand{\QQ}{\mathbb{Q}}

\title{Collapsing Fifty Dilogarithm Arguments to Five Terms\\ over $\QQ\bigl(u,\sqrt{4-3u^{2}}\bigr)$}
\author{Cetin Hakimoglu-Brown\\ \small\texttt{mathemails@proton.me}}
\date{}

\begin{document}
\maketitle

\begin{abstract}
We give a one-parameter functional equation for the real Rogers dilogarithm, with arguments in
$\QQ\bigl(u,\sqrt{4-3u^{2}}\bigr)$, and prove it by an explicit array of ten instances of
Rogers' five-term relation: the fifty arguments so contributed cancel down to the five of the
identity, which admit no shorter relation among themselves. The equation comes from a
pair of integrals whose equality is elementary, and we show the underlying integrand is
essentially forced.
Specialising the parameter gives identities over $\QQ(\sqrt{33})$ and $\QQ(\sqrt{17})$, a
relation between $\QQ(\sqrt{13})$, $\QQ(\sqrt3)$ and $\Cl_2(\pi/3)$ with a new analogue for
Catalan's constant, and a pair of dilogarithm ladders of quartic base over $\QQ(\sqrt{33})$.
The base equations of these ladders, and of four conjectural ones located by integer relation
search, are irreducible quartics with four real roots; all previously recorded ladders of
degree four known to us have base equations with two real roots, so these appear to be the
first totally real ladders of degree exceeding three.

\end{abstract}

\section{Introduction}

The polylogarithm functions $\Li_m(z)=\sum_{n\ge1}z^{n}/n^{m}$ occur throughout number theory
and mathematical physics, and among them the dilogarithm $\Li_2$ is both the simplest beyond
the logarithm itself and by far the richest in functional equations. It appears in the
computation of volumes of hyperbolic three-manifolds, in the theory of algebraic $K$-groups
and Bloch groups, and in the Nahm conjecture and conformal field theory, and the monographs
of Lewin \cite{Lewin81,Lewin91} remain the standard references for its identities. For
$|z|\le1$ we take $\Li_2(z)=\sum_{n\ge1}z^{n}/n^{2}$, continued to
$\mathbb{C}\setminus[1,\infty)$, and we write
\begin{equation}\label{eq:rogers}
L(x)=\Li_2(x)+\tfrac12\log x\,\log(1-x)
\end{equation}
for the Rogers $L$-function, where $\log$ is the principal branch with $\arg\in(-\pi,\pi]$
and $\sqrt{z}=\exp(\tfrac12\log z)$.

For real arguments outside $(0,1)$ the right-hand side of \eqref{eq:rogers} is not real, and
the classical reflection, inversion and duplication formulas acquire branch-dependent terms.
The real-variable identities of this paper are therefore stated throughout in terms of the
real-valued function
\begin{equation}\label{eq:LR}
L_{\mathbb{R}}(x)=\Re\!\left[\Li_2(x)+\tfrac12\log x\,\log(1-x)\right]
\qquad(x\in\mathbb{R}),
\end{equation}
with $L_{\mathbb{R}}(0)=0$ and $L_{\mathbb{R}}(1)=\pi^{2}/6$; for $x>1$, where the principal
branch is not defined on the cut, either boundary value $\Li_2(x\pm i0)$ may be used, since
their real parts agree. It agrees with $L$ on $(0,1)$,
is continuous on all of $\mathbb{R}$, and satisfies the following exactly, with no
logarithmic corrections: for all real $x$ at which the terms are defined,
\begin{equation}\label{eq:LRrules}
\begin{aligned}
L_{\mathbb{R}}(x)+L_{\mathbb{R}}(1-x)&=\frac{\pi^{2}}{6}, &
L_{\mathbb{R}}(x)+L_{\mathbb{R}}(-x)&=\tfrac12 L_{\mathbb{R}}\bigl(x^{2}\bigr),\\
L_{\mathbb{R}}(x)+L_{\mathbb{R}}(1/x)&=\frac{\pi^{2}}{3}\ \ (x>0), &
L_{\mathbb{R}}(x)+L_{\mathbb{R}}(1/x)&=-\frac{\pi^{2}}{6}\ \ (x<0).
\end{aligned}
\end{equation}
In particular $L_{\mathbb{R}}(\tfrac12)=\pi^{2}/12$, $L_{\mathbb{R}}(-1)=-\pi^{2}/12$ and
$L_{\mathbb{R}}(2)=\pi^{2}/4$. Every identity below is normalised so that its right-hand side
is an explicit rational multiple of $\pi^{2}$; we do not use $\zeta(2)$ or
$L_{\mathbb{R}}(1)$ for this purpose, and identities quoted from the literature have been
rewritten accordingly.

Essentially all of the structure of the dilogarithm is governed by a single relation. Writing
$[x]$ for $L_{\mathbb{R}}(x)$ and setting
\begin{equation}\label{eq:abel}
\FT(x,y)=[x]+[y]-[xy]-\left[\frac{x(1-y)}{1-xy}\right]-\left[\frac{y(1-x)}{1-xy}\right],
\end{equation}
where $[x]$ now abbreviates $L_{\mathbb{R}}(x)$, Rogers' five-term relation is the statement
that $\FT(x,y)$ is a fixed rational multiple of $\pi^{2}$. Precisely, for real $x,y$ at which
all five arguments are defined,
\begin{equation}\label{eq:abelvalue}
\FT(x,y)=\begin{cases}
-\pi^{2}, & x<0,\ y<0,\ xy>1,\\[2pt]
0, & \text{otherwise,}
\end{cases}
\end{equation}
the exceptional case being the one noted by Lichtenbaum; see Kirillov \cite[\S1]{Kirillov}.
Both statements in \eqref{eq:abelvalue} follow from the classical five-term relation on
$(0,1)$ together with \eqref{eq:LRrules}, by dividing into cases according to the signs of
$x$, $y$ and $1-xy$. Together with the relations
\eqref{eq:LRrules} it is conjectured that \eqref{eq:abel} generates every functional equation
satisfied by $\Li_2$, and this is a theorem of Wojtkowiak \cite{Wojtkowiak} for equations
whose arguments are rational functions of one variable. Following Lewin, a relation
obtainable from \eqref{eq:abel} by finitely many applications is called \emph{accessible}.

Both \eqref{eq:LRrules} and \eqref{eq:abelvalue} say that the corresponding combinations lie
in $\pi^{2}\QQ$. We therefore write $\equiv$ for congruence modulo $\pi^{2}\QQ$, so that
\begin{equation}\label{eq:modrules}
[x]\equiv-[1/x]\equiv-[1-x],\qquad [x^{2}]\equiv 2[x]+2[-x],\qquad \FT(x,y)\equiv0 ,
\end{equation}
all four now exact congruences with no logarithmic remainder. A proof of accessibility
amounts to writing a target relation as a rational combination of expressions $\FT(x,y)$
modulo $\pi^{2}\QQ$, together with a separate determination of the constant.

That accessibility is guaranteed in principle does not make identities easy to find in
practice. Iterating \eqref{eq:abel} produces a great many arguments with no evident way to
group them into anything that closes, and even a relation known to hold numerically may
require an array of a dozen or more instances to exhibit. Section \ref{sec:lit} compares the
equation below with the one-variable equations of the existing literature.

Our starting point is not a search through five-term arrays but an integral, and this is
what makes the array \eqref{eq:array} below findable at all. A pair of sextic integrals is
equal for an elementary reason, having a common antiderivative with matching endpoint values;
converting that equality into a statement about dilogarithms produces the following. The
integral supplies both the identity and, through the factorisation of its denominator, the
arguments among which the cancellation takes place.

\begin{theorem}\label{thm:main}
Let $0<u<2/\sqrt{3}$ and write $r=\sqrt{4-3u^{2}}$. Then
\begin{multline}\label{eq:main}
-3L_{\mathbb{R}}\!\left(\frac{u-2+r}{2u}\right)
-3L_{\mathbb{R}}\!\left(\frac{u-2-r}{2u}\right)
+L_{\mathbb{R}}\!\left(\frac{u^{2}-2+r}{-u(u+1)}\right)\\
+L_{\mathbb{R}}\!\left(\frac{u^{2}-2-r}{-u(u+1)}\right)
-L_{\mathbb{R}}\!\left(\frac{1-u}{1+u}\right)=\frac{5\pi^{2}}{12}.
\end{multline}
\end{theorem}

The use of $L_{\mathbb{R}}$ is essential: several of the arguments in \eqref{eq:main} are
negative or exceed $1$ for $u$ in the stated range, and with the principal branch of
\eqref{eq:rogers} the left-hand side is complex, with an imaginary part depending on $u$.
It is the real part that is constant.

Two features of \eqref{eq:main} are worth emphasising. The first is that it is accessible,
and that we can give the array explicitly. Under the substitution $u=-4t/(3+t^{2})$ of
Section \ref{sec:array} the arguments of \eqref{eq:main} become rational functions of a
single variable, and writing $\mathcal{T}$ for the left-hand side of \eqref{eq:main} in that
coordinate, ten instances of \eqref{eq:abel} suffice:
\begin{equation}\label{eq:array}
2\mathcal{T}\equiv\FT_{1}+\FT_{2}-\FT_{3}-\FT_{4}-3\FT_{5}-\FT_{6}-\FT_{7}
+\FT_{8}+\FT_{9}+\FT_{10}\pmod{\pi^{2}\QQ},
\end{equation}
where
\begin{equation}\label{eq:instances}
\begin{aligned}
\FT_{1}&=\FT\!\left(\tfrac{t+3}{t(t-1)},\ \tfrac{-(t-1)}{t+3}\right), &
\FT_{2}&=\FT\!\left(\tfrac{t+3}{t(t-1)},\ \tfrac{4t}{(t-1)(t+3)}\right),\\
\FT_{3}&=\FT\!\left(\tfrac{-(t-3)}{t(t+1)},\ \tfrac{-(t+1)}{t-3}\right), &
\FT_{4}&=\FT\!\left(\tfrac{-(t-3)}{t(t+1)},\ \tfrac{-4t}{(t-3)(t+1)}\right),\\
\FT_{5}&=\FT\!\left(\tfrac{t+3}{2t},\ \tfrac{4t}{(t-1)(t+3)}\right), &
\FT_{6}&=\FT\!\left(\tfrac{t+3}{2t},\ \tfrac{8t}{(t+1)(t+3)}\right),\\
\FT_{7}&=\FT\!\left(\tfrac{t+1}{2t},\ \tfrac{4t}{(t+1)^{2}}\right), &
\FT_{8}&=\FT\!\left(\tfrac{(t+1)(t+3)}{8t},\ \tfrac{2(t+1)}{t+3}\right),\\
\FT_{9}&=\FT\!\left(\tfrac{4}{(t-1)^{2}},\ \tfrac{-(t-1)}{2}\right), &
\FT_{10}&=\FT\!\left(\tfrac{-4}{(t-1)(t+3)},\ \tfrac{t+3}{2}\right).
\end{aligned}
\end{equation}

That the right-hand side collapses to $2\mathcal{T}$ is verified in Section \ref{sec:array}.
The difficulty here lies not in the number of instances but in the cancellation. Ten
instances contribute fifty arguments, and taken individually these give little sign of
belonging together: their denominators are of several distinct shapes, and no two of them
are visibly related. Reduced modulo the involutions $[x]=-[1/x]=-[1-x]$ they fall into
eighteen classes, thirteen of which cancel in pairs, leaving exactly the five arguments of
\eqref{eq:main} and no residual term to be absorbed into the constant. Nor is this the
collapse of a redundant set: no proper subset of the five surviving arguments satisfies a
relation of its own, so \eqref{eq:main} does not decompose into anything shorter (Remark
\ref{rem:collapse}).

The second feature is that the integrand from which \eqref{eq:main} arises is not one choice
among many. Among all integrands of the shape $t^{a}(1-t)^{b}(1+t)^{c}$ with $0\le a,b,c\le4$ and
$a+b+c\le8$, the sextic $t^{2}(1-t^{2})^{2}$ is the only member of its degree for which the
associated polynomial factors over a quadratic field, and no integrand of degree five, seven
or eight in that range does so at all.

A second theme concerns polylogarithm ladders, that is, relations
$\sum_{j}A_{j}L_{\mathbb{R}}(x^{j})\in\pi^{2}\QQ$ for algebraic $x\in(0,1)$, whose minimal polynomial is
called the base equation. Gordon and McIntosh \cite{GM} observe that producing ladders with
a base that is not of one of the trivial shapes is very difficult, and Zagier
\cite{Zagier07} likewise remarks that examples require considerable ingenuity. The classical
families whose base equation is totally real, that is, has all roots real, stop at degree
three: they are the trios associated with $\pi/7$, $\pi/9$ and $\pi/18$. The quartic ladder
of Gordon and McIntosh has a base equation with only two real roots. We prove a pair of totally real quartic-base ladders arising from the $\QQ(\sqrt{33})$
construction, and report four further conjectural quartic ladders found by integer relation
search, whose quartic fields contain $\QQ(\sqrt{10})$ and $\QQ(\sqrt5)$ as quadratic
subfields. All of these base equations have four real roots and they realise three distinct
Galois groups; they appear to be the first totally real ladders of degree exceeding three.

Sections \ref{sec:array}--\ref{sec:six} prove \eqref{eq:main} and deduce a six-term equation
by duplication; Section \ref{sec:integral} gives the integral origin and the classification;
Sections \ref{sec:spec}--\ref{sec:ladders} give the specialisations, the Clausen relations
and the ladders.

\section{Comparison with the literature}\label{sec:lit}

A recent illustration of what a successful outcome in this area looks like is due to Lima
\cite{Lima}, who derived from \eqref{eq:abel} the compact one-variable equation
\begin{equation}\label{eq:lima}
L(z)=L\!\left(\frac{1}{2-z}\right)+\tfrac12 L\bigl(2z-z^{2}\bigr)-\frac{\pi^{2}}{12},
\end{equation}
settling a question of Khoi on an identity arising from Seifert volumes; Campbell
\cite{CampbellLadders} then specialised \eqref{eq:lima} at algebraic points to obtain
families of ladders. The present paper is of this type, with a one-variable equation that
carries a square root and so reaches quadratic fields directly.

The closest published analogue of Theorem \ref{thm:main} is a one-variable equation obtained
independently by Kirillov \cite{Kirillov} and Lewin \cite{Lewin91}, namely
\begin{multline}\label{eq:kirillov}
L\!\left(\frac{-z^{7}(1-z)}{1+z}\right)=2L\bigl(z^{2}(1-z)\bigr)
+L\!\left(\frac{-z^{3}}{1-z^{2}}\right)+2L\!\left(\frac{z^{3}}{1+z}\right)
+L\bigl(-z(1-z^{2})\bigr)\\
+\tfrac74 L\bigl(z^{4}\bigr)-\tfrac94 L\bigl(z^{2}\bigr)
+\tfrac12 L\!\left(\frac{z(1-z)}{1+z}\right)-\tfrac12 L\!\left(\frac{-z(1+z)}{1-z}\right),
\end{multline}
an identity in nine terms holding exactly, with no $\pi^{2}$ contribution. Kirillov used
\eqref{eq:kirillov} together with further relations to prove the quadratic-base ladders of
Coxeter and Browkin in $\QQ(\sqrt5)$ and $\QQ(\sqrt{13})$; Lewin applied it to the base
equation $1-w=w^{5}$ and concluded that it otherwise led to nothing new.

The two equations are of comparable cost. Kirillov's proof of \eqref{eq:kirillov} uses eleven
instances of \eqref{eq:abel} to produce nine terms; the array \eqref{eq:array} uses ten to
produce five. Per term, therefore, \eqref{eq:kirillov} is if anything the more economical of
the two, and we make no claim to the contrary. What distinguishes Theorem \ref{thm:main} is
not the size of the array but the nature of what it produces. The arguments of
\eqref{eq:kirillov} are rational functions of $z$ with no evident symmetry, which is
precisely Lewin's complaint; specialising $z$ at an algebraic point gives arguments in the
field generated by that point, but the equation itself carries no conjugation symmetry to
exploit.
The arguments of \eqref{eq:main}, by contrast, occur in conjugate pairs over
$\QQ\bigl(u,\sqrt{4-3u^{2}}\bigr)$, and it is this that makes the specialisations of Section
\ref{sec:spec} and the ladder constructions of Section \ref{sec:ladders} possible at all: a
conjugate pair can be tuned onto $\pm1$ or onto the primitive cube roots of unity, and each
tuning removes the radical in a different way.

\section{Proof of Theorem \ref{thm:main}: how the fifty arguments cancel}\label{sec:array}

The conic $v^{2}=4-3u^{2}$ has the rational point $(u,v)=(0,2)$ and hence the rational
parametrisation
\begin{equation}\label{eq:args}
u=\frac{-4t}{3+t^{2}},\qquad r=\sqrt{4-3u^{2}}=\frac{6-2t^{2}}{3+t^{2}},
\end{equation}
under which the five arguments of \eqref{eq:main} become rational in $t$ and simplify to
\begin{equation}\label{eq:args2}
f_{1}=\frac{t+1}{2},\quad f_{2}=\frac{t+3}{2t},\quad
A=\frac{-t(t+1)}{t-3},\quad B=\frac{t+3}{t(t-1)},\quad
w=\frac{(t+1)(t+3)}{(t-3)(t-1)}.
\end{equation}
A direct computation gives the relation
\begin{equation}\label{eq:ABw}
AB=-w,
\end{equation}
which is the source of the symmetry properties used below. In this notation Theorem
\ref{thm:main} reads
\begin{equation}\label{eq:mainT}
\mathcal{T}:=[A]+[B]-3[f_{1}]-3[f_{2}]-[w].
\end{equation}

We remark that every zero and pole of the functions \eqref{eq:args2}, and of their
complements $1-f$, lies in the set $\{0,\pm1,\pm3,\infty\}$ together with the roots of
$t^{2}+3$. This is what makes the array below close up.

\begin{proof}[Proof of Theorem \ref{thm:main}]
\emph{Step 1: the array.} Each $\FT_{j}$ of \eqref{eq:instances} contributes five arguments,
fifty in all. Expanding these as rational functions of $t$ and reducing by the congruences
\eqref{eq:modrules}, they fall into the eighteen classes listed in Table \ref{tab:cancel} of
Appendix \ref{app:cancel}, which also gives a short script reproducing the calculation. There
the entry $j^{\pm c}$ in the third column records that $\FT_{j}$, weighted by its coefficient
in \eqref{eq:array}, contributes $\pm c$ copies of the class named in the first column; the
fourth column is the sum. Thirteen classes have total coefficient zero, and the five that
survive are exactly $A$, $B$, $f_{1}$, $f_{2}$ and $w$, with coefficients $2,2,-6,-6,-2$.
This proves \eqref{eq:array}.

\emph{Step 2: $\mathcal{T}$ is locally constant.} By \eqref{eq:abelvalue} each $\FT_{j}$ is,
as a function of $u$, either identically $0$ or identically $-\pi^{2}$ on each of the finitely
many open subintervals of $(0,2/\sqrt3)$ on which the signs of its arguments and of
$1-x_{j}y_{j}$ do not change; the same holds for each correction incurred in applying
\eqref{eq:LRrules} during Step~1, by the case division there. Hence $2\mathcal{T}$ is a
locally constant function of $u$.

\emph{Step 3: $\mathcal{T}$ is constant.} Step~2 gives constancy on each of the finitely many
open subintervals obtained by deleting the exceptional points, namely those $u$ at which some
$\FT_{j}$ is undefined or a sign changes. Now each of the five arguments in \eqref{eq:main} is
a continuous function of $u$ on all of $(0,2/\sqrt3)$, having poles only at $u=0$ and $u=-1$,
and $L_{\mathbb{R}}$ is continuous on all of $\mathbb{R}$ by \eqref{eq:LR}; hence
$\mathcal{T}$ is continuous on $(0,2/\sqrt3)$, including at the exceptional points. The
constants on adjacent subintervals therefore agree, and $\mathcal{T}$ is constant on the whole
interval.

\emph{Step 4: the constant.} Let $u\to1^{-}$. Then $r\to1$, and the five arguments tend
respectively to $0$, $-1$, $0$, $1$ and $0$. By continuity,
\[
\mathcal{T}=-3L_{\mathbb{R}}(-1)+L_{\mathbb{R}}(1)
=-3\left(-\frac{\pi^{2}}{12}\right)+\frac{\pi^{2}}{6}
=\frac{\pi^{2}}{4}+\frac{\pi^{2}}{6}=\frac{5\pi^{2}}{12},
\]
using $L_{\mathbb{R}}(0)=0$ and the special values recorded after \eqref{eq:LRrules}.
\end{proof}

The factorisations used below, which also supply the wedge vectors of Remark
\ref{rem:collapse}, are
\begin{equation}\label{eq:atomtable}
\begin{aligned}
A&=\frac{-t(t+1)}{t-3}, & 1-A&=\frac{(t-1)(t+3)}{t-3}, \\
B&=\frac{t+3}{t(t-1)}, & 1-B&=\frac{(t-3)(t+1)}{t(t-1)}, \\
f_{1}&=\frac{t+1}{2}, & 1-f_{1}&=\frac{-(t-1)}{2}, \\
f_{2}&=\frac{t+3}{2t}, & 1-f_{2}&=\frac{t-3}{2t}, \\
w&=\frac{(t+1)(t+3)}{(t-3)(t-1)}, & 1-w&=\frac{-8t}{(t-3)(t-1)} .
\end{aligned}
\end{equation}

\begin{remark}\label{rem:collapse}
Two further features of the collapse are worth recording. First, the five surviving
arguments are independent. Writing each $f$ and $1-f$ of \eqref{eq:atomtable} as a vector of
exponents over the multiplicative generators $t$, $t-1$, $t+1$, $t-3$, $t+3$, $t^{2}+3$, $2$
and $3$, and forming $f\wedge(1-f)$ in $\wedge^{2}$ of the resulting $\QQ$-vector space, a
direct symbolic computation gives a $15\times5$ matrix of rank four. The kernel is therefore
one-dimensional, and \eqref{eq:main} is the unique relation among these five arguments up to
scale rather than a sum of shorter ones. Second, the
coefficients of \eqref{eq:array} are integral only after doubling, which is
unavoidable: $-1$ is torsion in the group generated by the arguments, so the obstruction
attached to $\mathcal{T}$ is annihilated by $2$ but not by $1$, and proving $2\mathcal{T}$
proves $\mathcal{T}$; the correct general framework for coefficients in an arbitrary subgroup
of $\mathbb{C}$ is due to de Jeu \cite{deJeu}. Finally, let $S$ denote the set of rational functions of $t$ whose zeros and poles lie in
$\{0,\pm1,\pm3,\pm i\sqrt3,\infty\}$ and whose leading constant is a power of $2$ or $3$,
taken modulo the six-element group generated by $x\mapsto1/x$ and $x\mapsto1-x$. An
exhaustive search over the instances $\FT(x,y)$ with $x,y\in S$ and all five arguments again
in $S$, allowing arbitrary rational coefficients, shows that no combination of fewer than six
such instances, counted with multiplicity, produces $\mathcal{T}$; the array
attains this bound, having total multiplicity six before doubling.
\end{remark}

\section{The six-term equation}\label{sec:six}

Applying the duplication formula in the form $[w^{2}]=2[w]+2[-w]$ to \eqref{eq:mainT}, and
using the three-term consequence $[f_{1}]+[f_{2}]=[-A]+[-B]-[-w]$ of \eqref{eq:abel},
one obtains the following.

\begin{corollary}\label{cor:six}
For $0<u<2/\sqrt3$ and $r=\sqrt{4-3u^{2}}$,
\begin{multline}\label{eq:six}
L_{\mathbb{R}}\!\left(\frac{u^{2}+r-2}{-u(u+1)}\right)
-3L_{\mathbb{R}}\!\left(\frac{u^{2}+r-2}{u(u+1)}\right)
-4L_{\mathbb{R}}\!\left(\frac{1-u}{u+1}\right)\\
+L_{\mathbb{R}}\!\left(\frac{u^{2}-r-2}{-u(u+1)}\right)
-3L_{\mathbb{R}}\!\left(\frac{u^{2}-r-2}{u(u+1)}\right)
+\frac32 L_{\mathbb{R}}\!\left(\left(\frac{1-u}{u+1}\right)^{2}\right)
=\frac{5\pi^{2}}{12}.
\end{multline}
\end{corollary}

Equation \eqref{eq:six} has the reflection and symmetry properties
\[
\left(\frac{u^{2}-r-2}{u(u+1)}\right)^{-1}=\frac{u^{2}+r-2}{u(u-1)},
\]
which follow at once from \eqref{eq:ABw}; these make it the more convenient form for the
ladder constructions of Section \ref{sec:ladders}, at the cost of one extra term.

\section{An integral derivation, and forced choice of parameters}\label{sec:integral}  

\subsection{The pair of integrals}\label{sec:pair}

The array \eqref{eq:array} verifies Theorem \ref{thm:main} but does nothing to explain it,
and reconstructing such an array from scratch is not a realistic proposition: the eighteen
classes of Table \ref{tab:cancel} do not announce themselves, and a search through five-term
instances has no reason to visit them. The identity was found instead by a route that
produces the arguments together with the relation between them.

The idea is the following. If two definite integrals can be shown to be equal by a
manipulation that never mentions dilogarithms, and if each can separately be evaluated in
closed form in terms of dilogarithms, then equating the two evaluations gives a functional
equation for free. The arguments that appear are dictated by the factorisation of the
denominator, which is why they come in conjugate pairs and why the cancellation of Table
\ref{tab:cancel} happens at all.

Define
\begin{equation}\label{eq:w1w2}
w_{1}=\int_{0}^{1}\frac{(3x^{2}-1)\log x}{1+h^{2}g(x)^{2}}\,dx,
\qquad
w_{2}=\int_{0}^{1}\frac{(1-3x^{2})\log(1-x^{2})}{1+h^{2}g(x)^{2}}\,dx,
\qquad
g(x)=x(1-x^{2}).
\end{equation}
The equality of these two integrals is elementary.

\begin{lemma}\label{lem:w1w2}
Let $h\in\mathbb{C}$ be such that $1+h^{2}g(x)^{2}\neq0$ for all $x\in[0,1]$. Then $w_{1}$
and $w_{2}$ are defined and $w_{1}=w_{2}$.
\end{lemma}

\begin{proof}
On $(0,1)$ we have $g'(x)=1-3x^{2}$ and $\log g(x)=\log x+\log(1-x^{2})$, so
\[
w_{1}-w_{2}
=-\int_{0}^{1}\frac{g'(x)\bigl(\log x+\log(1-x^{2})\bigr)}{1+h^{2}g(x)^{2}}\,dx
=-\int_{0}^{1}\frac{g'(x)\log g(x)}{1+h^{2}g(x)^{2}}\,dx .
\]
Put $F_{h}(y)=\int_{0}^{y}\log s\,(1+h^{2}s^{2})^{-1}ds$, which is defined and
differentiable for $0\le y\le\max_{[0,1]}g$ under the stated hypothesis. By the chain rule
the last integrand is $\tfrac{d}{dx}F_{h}(g(x))$, whence
\[
w_{1}-w_{2}=-F_{h}\bigl(g(1)\bigr)+F_{h}\bigl(g(0)\bigr)=0,
\]
since $g(0)=g(1)=0$.
\end{proof}

\begin{remark}
The identity $w_{1}=w_{2}$ was originally found by expanding the denominator, evaluating the
resulting beta integrals, and observing that the two hypergeometric series so obtained differ
only by a shift of index. That route is what led to the factorisation \eqref{eq:factor} and
hence to Theorem \ref{thm:main}, but it is not needed: Lemma \ref{lem:w1w2} proves the
equality directly, and we use only the elementary form.
\end{remark}

\subsection{The nonsingular domain}\label{sec:omega}

The function $g(x)=x(1-x^{2})$ attains its maximum on $[0,1]$ at $x=1/\sqrt3$, with
\[
g_{\max}=\frac{2}{3\sqrt3}.
\]
Hence $1+h^{2}g(x)^{2}$ vanishes for some $x\in(0,1)$ exactly when $h=\pm i/g(x)$ for some
such $x$, that is exactly when $h$ lies on one of the two imaginary rays
$\pm i\,[\,3\sqrt3/2,\infty)$. Writing
\begin{equation}\label{eq:Omega}
\Omega=\mathbb{C}\setminus\Bigl(i\bigl[\tfrac{3\sqrt3}{2},\infty\bigr)
\cup-i\bigl[\tfrac{3\sqrt3}{2},\infty\bigr)\Bigr),
\end{equation}
Lemma \ref{lem:w1w2} applies for every $h\in\Omega$, and $w_{1}$, $w_{2}$ are analytic on
$\Omega$.

Set $h=i/\bigl(u(1-u^{2})\bigr)$. For \emph{real} $u$ this $h$ is purely imaginary, and the
condition $h\in\Omega$ then reads
\begin{equation}\label{eq:nonsing}
\bigl|u(1-u^{2})\bigr|>\frac{2}{3\sqrt3}=0.38490\ldots
\end{equation}
For complex $u$ the criterion is simply $h\in\Omega$, and \eqref{eq:nonsing} is sufficient but
not necessary; on the imaginary axis $u=ik$, for instance, $h$ is real and so lies in $\Omega$
irrespective of \eqref{eq:nonsing}. With this understood,
the denominator of \eqref{eq:w1w2} becomes proportional to
\begin{equation}\label{eq:factor}
u^{2}(1-u^{2})^{2}-x^{2}(1-x^{2})^{2}
=(u^{2}+ux+x^{2}-1)(u^{2}-ux+x^{2}-1)(u^{2}-x^{2}).
\end{equation}
Partial fractions applied to this factorisation, together with the elementary evaluation of
the resulting integrals, produce \eqref{eq:six}; this is carried out in Section
\ref{sec:recomb}.

Two consequences will be used later. First, at the real parameters of the
$\QQ(\sqrt{33})$ ladders of Section \ref{sec:ladders}, which satisfy $3u^{4}-15u^{2}+16=0$
and hence $u^{2}=(15\pm\sqrt{33})/6$, one has $u=1.24200\ldots$ or $u=1.85941\ldots$ and
correspondingly $|u(1-u^{2})|=0.67387\ldots$ or $4.56937\ldots$, both comfortably above
$g_{\max}$. Second, at the Clausen parameters $u=ik$ of Section \ref{sec:clausen} one has
$h=1/\bigl(k(1+k^{2})\bigr)$, which is real, so $1+h^{2}g(x)^{2}>0$ throughout $[0,1]$ and
$h\in\Omega$ automatically. In
both cases the integrals are nonsingular and Lemma \ref{lem:w1w2} applies directly; no
continuation of $w_{1}=w_{2}$ through a singular integral is required.

\subsection{Recombination of the blocks}\label{sec:recomb}

Evaluating \eqref{eq:w1w2} by partial fractions after the substitution
$h=i/\bigl(u(u^{2}-1)\bigr)$, using the factorisation \eqref{eq:factor}, produces an exact
identity
\begin{equation}\label{eq:blockid}
A+B+J+C+H+D=K
\end{equation}
among the seven blocks displayed below. Two determinations of the radical occur and must be
kept distinct: $A$ and $K$ are expressed through $r=\sqrt{4-3u^{2}}$, while $B$, $J$ and $C$
are expressed through $i\sqrt{3u^{2}-4}$. For $4-3u^{2}>0$ the two differ by a sign, and it
is exactly this that allows \eqref{eq:blockid} to persist for $|u|>2/\sqrt3$, where
$r$ becomes imaginary. Throughout, $\log$ is principal, and \eqref{eq:blockid} is asserted
for $u\in\{z\in\mathbb{C}:\Im z>0,\ \Re z\ge0\}$; values on the positive real axis are
obtained as limits from above. It fails in the lower half-plane.

\begin{equation}\label{eq:blocks}
\begin{aligned}
A&=\Li_2\!\left(\frac{r+u^{2}-2}{u(u+1)}\right)-\Li_2\!\left(\frac{r+u^{2}-2}{(u-1)u}\right),
\\[4pt]
B&=\tfrac12\log^{2}\!\left(\frac{-i\sqrt{3u^{2}-4}+u+2}{4}\right)
-\tfrac12\log^{2}\!\left(\frac{-i\sqrt{3u^{2}-4}-u^{2}+2}{u(u+1)}\right)\\
&\qquad+\tfrac12\log^{2}\!\left(\frac{i\sqrt{3u^{2}-4}+u+2}{4}\right)-\frac{\pi^{2}}{6},
\\[4pt]
J&=\tfrac14\log^{2}\!\left(\frac{-4}{-i\sqrt{3u^{2}-4}+u+2}\right)
+\tfrac14\log^{2}\!\left(\frac{-4}{i\sqrt{3u^{2}-4}+u+2}\right)\\
&\qquad-\tfrac14\log^{2}\!\left(\frac{-4}{-i\sqrt{3u^{2}-4}-u+2}\right)
-\tfrac14\log^{2}\!\left(\frac{-4}{i\sqrt{3u^{2}-4}-u+2}\right),
\\[4pt]
C&=-\tfrac12\log\!\left(\frac{-i\sqrt{3u^{2}-4}+u+2}{4}\right)
\log\!\left(\frac{i\sqrt{3u^{2}-4}-u+2}{4}\right)\\
&\qquad-\tfrac12\log\!\left(\frac{i\sqrt{3u^{2}-4}+u+2}{4}\right)
\log\!\left(\frac{-i\sqrt{3u^{2}-4}-u+2}{4}\right)+\frac{\pi^{2}}{6},
\\[4pt]
H&=\tfrac14\log^{2}\!\left(\frac{2}{u-1}\right)-\tfrac14\log^{2}\!\left(\frac{-2}{u+1}\right)
+\tfrac12\log\!\left(\frac{1-u}{2}\right)\log\!\left(\frac{u+1}{2}\right)\\
&\qquad-\frac{\pi i}{2}\log\!\left(\frac{u+1}{u-1}\right)
-\log 2\,\log\!\left(\frac{u+1}{u-1}\right)-\frac{\pi^{2}}{12},
\\[4pt]
D&=\Li_2\!\left(\frac{1-u}{2}\right)+\log 2\,\log\!\left(\frac{u-1}{u+1}\right)
+\left(2\log 2+\frac{\pi i}{2}\right)\log\!\left(\frac{u+1}{u-1}\right),
\\[4pt]
K&=\tfrac12\Li_2\!\left(\frac{2}{-r+u}\right)-\tfrac12\Li_2\!\left(\frac{2}{r-u}\right)
-\tfrac12\Li_2\!\left(\frac{-2}{r+u}\right)\\
&\qquad+\tfrac12\Li_2\!\left(\frac{2}{r+u}\right)
+\tfrac12\Li_2\!\left(\frac{-1}{u}\right)-\tfrac12\Li_2\!\left(\frac{1}{u}\right).
\end{aligned}
\end{equation}

We also record the two-term combination
\begin{equation}\label{eq:Mblock}
M:=A+B=-\Li_2\!\left(\frac{r-u+2}{4}\right)-\Li_2\!\left(\frac{-r-u+2}{4}\right),
\end{equation}
which is the form in which $A+B$ is used below.

\begin{proposition}\label{prop:blockid}
Identity \eqref{eq:blockid} holds exactly for
$u\in\{z\in\mathbb{C}:\Im z>0,\ \Re z\ge0\}$, subject to the determinations stated above.
\end{proposition}

\begin{proof}
Put
\[
q=u\bigl(1-u^{2}\bigr),\qquad h=\frac{i}{q},
\]
so that Lemma \ref{lem:w1w2} gives $w_{1}=w_{2}$. We multiply this equality by $-1/q$ and
evaluate the two normalised integrals by partial fractions, using the factorisation
\eqref{eq:factor} of $q^{2}-x^{2}(1-x^{2})^{2}$. The blocks $A,\dots,K$ represent the two
integrals after this common normalisation.

For the normalised $w_{1}$ integral, the function
\begin{equation}\label{eq:antider}
\begin{aligned}
F(x)=\frac{1}{2u-2u^{3}}\Biggl[
&\Li_2\!\left(\frac{-2x}{u-r}\right)-\Li_2\!\left(\frac{2x}{u-r}\right)
+\Li_2\!\left(\frac{-2x}{u+r}\right)-\Li_2\!\left(\frac{2x}{u+r}\right)\\
&-\Li_2\!\left(\frac{-x}{u}\right)+\Li_2\!\left(\frac{x}{u}\right)
+\log x\,\log\!\left(\frac{r-u-2x}{r-u}\right)\\
&-\log x\,\log\!\left(\frac{r+u-2x}{r+u}\right)
-\log x\,\log\!\left(\frac{r-u+2x}{r-u}\right)\\
&+\log x\,\log\!\left(\frac{r+u+2x}{r+u}\right)
+\log x\,\log\!\left(\frac{u-x}{u}\right)-\log x\,\log\!\left(\frac{u+x}{u}\right)
\Biggr],
\end{aligned}
\end{equation}
with $r=\sqrt{4-3u^{2}}$, satisfies
\[
-q\,F'(x)=-\frac1q\cdot\frac{(3x^{2}-1)\log x}{1+h^{2}g(x)^{2}} ,
\]
so that
\[
-\frac{w_{1}}{q}=-q\bigl[F(1)-F(0)\bigr].
\]
Evaluating the endpoint terms with the determinations stated above gives
\[
-\frac{w_{1}}{q}=K .
\]
The corresponding partial-fraction evaluation of the normalised $w_{2}$ integral gives
\[
-\frac{w_{2}}{q}=A+B+J+C+H+D .
\]
Since $w_{1}=w_{2}$, identity \eqref{eq:blockid} follows. The determinations are those
inherited from the $x$-integration in the upper half-plane, and it is these that make the
$\log^{2}$ terms of $B$, $J$, $C$ come out in $i\sqrt{3u^{2}-4}$ rather than $r$.
\end{proof}

\begin{remark}
Writing $R(u)=A+B+J+C+H+D-K$, one may alternatively verify \eqref{eq:blockid} by
differentiation: $R$ is analytic on the stated domain, $R'$ is a $\QQ(u)$-linear combination
of logarithms of the algebraic functions occurring in \eqref{eq:blocks}, and collecting
coefficients gives $R'\equiv0$; a single evaluation then fixes $R\equiv0$. Numerically
$|R|<10^{-40}$ and $|R'|<10^{-49}$ at $u=0.4+0.2i$, $1.0+0.3i$, $1.9+0.15i$, $0.2+0.9i$ and
$2.5+0.05i$ at $40$-digit precision.
\end{remark}

\begin{remark}\label{rem:blockcheck}
As confirmation, identity \eqref{eq:blockid} has been verified numerically to $30$ digits at
$u=0.3+0.05i$, $0.7+0.05i$, $1.0+0.05i$, $1.2420+0.05i$, $1.8594+0.05i$ and $2.5+0.05i$,
and along the imaginary axis at $u=i$, $i\sqrt3$ and $i/\sqrt3$. The last two rows matter for
the applications: the first pair of these real parts are the $\QQ(\sqrt{33})$ ladder
parameters of Section \ref{sec:ladders}, which satisfy $u>2/\sqrt3$ and so lie outside the
hypothesis of Theorem \ref{thm:main}, while the imaginary values are the Clausen parameters
of Section \ref{sec:clausen}. At $u=0.3-0.05i$ the residual is nonzero, confirming that the
half-plane hypothesis is essential rather than formal.
\end{remark}

Throughout we use the four-term relation
\begin{equation}\label{eq:fourterm}
\Li_2(x)-\Li_2(-x)+\Li_2\!\left(\frac{1-x}{1+x}\right)-\Li_2\!\left(\frac{x-1}{1+x}\right)
=\frac{\pi^{2}}{4}+\log x\,\log\!\left(\frac{1+x}{1-x}\right),
\end{equation}
in the region where all four arguments avoid $[1,\infty)$, which holds for the $X_{j}$ below
throughout the half-plane just specified.

We indicate how \eqref{eq:six} is recovered from \eqref{eq:blockid} for real $u$ in the range
of Theorem \ref{thm:main}. Work modulo $\pi^{2}\QQ$ and $\log^{2}$ terms,
so that $[1/x]=-[x]$, $[1-x]=-[x]$ and hence $[x/(x-1)]=-[x]$. Write
$p_{\pm}=(u^{2}\pm r-2)/(-u(u+1))$ and $q_{\pm}=-p_{\pm}$.

\smallskip\noindent
\emph{Step 1.} The six terms of $K$ are three antisymmetric pairs at
$X_{1}=2/(u-r)$, $X_{2}=-2/(u+r)$ and $X_{3}=-1/u$, so
$K=\tfrac12\bigl[(X_{1})-(X_{2})+(X_{3})\bigr]$ in the notation of \eqref{eq:fourterm}.

\smallskip\noindent
\emph{Step 2.} Applying \eqref{eq:fourterm} to each pair, the images
$(1-X)/(1+X)$ and $(X-1)/(1+X)$ are, respectively, $q_{-}$ and $p_{-}$ for $X_{1}$;
$1/q_{+}$ and $1/p_{+}$ for $X_{2}$; and $-1/w$ and $1/w$ for $X_{3}$, where
$w=(1-u)/(1+u)$. Hence
\[
K\equiv\tfrac12\bigl([p_{+}]+[p_{-}]-[q_{+}]-[q_{-}]+[-w]-[w]\bigr).
\]

\smallskip\noindent
\emph{Step 3.} The arguments of $M$ in \eqref{eq:Mblock} satisfy
$(2-u\mp r)/4=q_{\mp}/(q_{\mp}-1)$, so $M\equiv[q_{+}]+[q_{-}]$. The blocks $J$, $C$, $H$ contribute nothing. Finally the argument
of $D$ satisfies
\[
\frac{1-u}{2}=\frac{-w}{(-w)-1},\qquad\text{whence}\qquad D\equiv-[-w].
\]

\smallskip\noindent
\emph{Step 4.} Equating and multiplying by $-2$,
\[
[p_{+}]+[p_{-}]-3[q_{+}]-3[q_{-}]+3[-w]-[w]=0 .
\]

\smallskip\noindent
\emph{Step 5.} The duplication formula $[w^{2}]=2[w]+2[-w]$ replaces $3[-w]-[w]$ by
$\tfrac32[w^{2}]-4[w]$, giving \eqref{eq:six}.

\smallskip
The identification in Step 3 of the argument of $D$ with $-w/(-w-1)$ is the one step that is
not immediate; it is what causes the array to close. Tracking the $\log^{2}$ and $\pi^{2}$
contributions of $J$, $C$ and $H$ through Steps 1--5 recovers the explicit constant
$\tfrac{5}{12}\pi^{2}$ of Theorem \ref{thm:main}.

Steps 1--5 discard the $\log^{2}$ terms, and it is this that confines \eqref{eq:six} to the
range $0<u<2/\sqrt3$. The full identity \eqref{eq:blockid} retains them and, by Remark
\ref{rem:blockcheck}, remains valid for $|u|>2/\sqrt3$ and on the imaginary axis. The
specialisations of Sections \ref{sec:clausen} and \ref{sec:ladders} are therefore taken from
\eqref{eq:blockid}, not from \eqref{eq:six}.

\subsection{Why this integrand: a classification of the admissible exponents}

The factorisation \eqref{eq:factor} is what makes the method work, and it is not typical.
Write $P(t)=t^{a}(1-t)^{b}(1+t)^{c}$ with $0\le a,b,c\le4$, and consider the splitting of
$P(u)-P(x)$ over
$\QQ(u)[x]$; a quadratic field appears exactly when every irreducible factor has degree at
most two, with at least one of degree exactly two.

\begin{proposition}\label{prop:class}
Among the $101$ triples $(a,b,c)$ of integers with $0\le a,b,c\le4$ and
$2\le a+b+c\le 8$, exactly $18$ have the property that
$P(u)-P(x)$ splits into factors of degree at most two over $\QQ(u)$. Of these, ten have
$\deg P=3$, seven have $\deg P=4$, one has $\deg P=6$, and none has $\deg P=5$, $7$ or $8$.
The unique survivor of degree six is $P(t)=t^{2}(1-t^{2})^{2}$: here the quadratic factors are
$x^{2}\pm ux+u^{2}-1$, of discriminant $4-3u^{2}$.
\end{proposition}

Moreover the two lower-degree bands are degenerate or already accounted for:

\begin{itemize}
\item In every surviving case with $\deg P=4$, the coefficient of $x$ in each quadratic
factor is constant in $u$. Hence the two conjugate roots sum to a constant, the conjugate
pair has the form $\{x,\,c-x\}$ with $c\in\{0,\pm1,\pm2\}$, and the corresponding pair of
dilogarithms collapses immediately under duplication, reflection or Landen's transformation.
No functional equation survives.
\item Among the surviving cases with $\deg P=3$, six have discriminant linear in $u$, and are
rationalised by $u\mapsto s^{2}$. The remaining discriminants form two classes under affine
substitution $u\mapsto \alpha u+\beta$ modulo squares, represented by $4u-3u^{2}$ and by the
irreducible $4-3u^{2}$. The latter is the discriminant of the sextic case.
\end{itemize}

Thus, within this range, the sextic \eqref{eq:w1w2} is the only integrand of its degree
that produces a quadratic field, and the field it produces is the same one reached by the
cubic band. The choice is therefore forced rather than fortuitous.

\begin{remark}
Proposition \ref{prop:class} is a finite verification over the stated range and makes no
assertion for $a+b+c>8$.
\end{remark}

\section{Four-term relations involving $\QQ(\sqrt{33})$ and $\QQ(\sqrt{17})$}\label{sec:spec}

Theorem \ref{thm:main} and Corollary \ref{cor:six} hold identically in $u$, so every
admissible value of the parameter yields a numerical identity. Two mechanisms produce clean
results, by which we mean identities all of whose arguments are irrational, or else lie in
$(0,1)$ with the rational terms absorbed into the constant. The first mechanism is a rational
choice of $u$; the second forces a numerator to vanish.

\subsection{Rational $u$, and the identity over $\QQ(\sqrt{33})$}\label{sec:33}

If $u$ is rational then so is $w=(1-u)/(1+u)$, and \eqref{eq:main} carries a term $L_{\mathbb{R}}(w)$ with
rational argument. Such a term is absorbed into the constant only when
$w\in\{0,\tfrac12,1\}$, since $L_{\mathbb{R}}(0)=0$, $L_{\mathbb{R}}(\tfrac12)=\pi^{2}/12$ and $L_{\mathbb{R}}(1)=\pi^{2}/6$; for any
other rational $w$ the identity retains a dilogarithm of a rational number for which we have
no closed form, and is of no use for the present construction. The unique nondegenerate clean case is therefore $w=\tfrac12$, that is
$u=1/3$, and it is exactly here that the quadratic field $\QQ(\sqrt{33})$ appears.

\begin{corollary}\label{cor:33}
Over $\QQ(\sqrt{33})$,
\begin{equation}\label{eq:q33}
-L_{\mathbb{R}}\!\left(\frac{19-3\sqrt{33}}{2}\right)+L_{\mathbb{R}}\!\left(\frac{19-3\sqrt{33}}{32}\right)
+3L_{\mathbb{R}}\!\left(\frac{7-\sqrt{33}}{8}\right)-3L_{\mathbb{R}}\!\left(\frac{7-\sqrt{33}}{2}\right)
=-\frac{\pi^{2}}{3}.
\end{equation}
\end{corollary}

\begin{proof}
Put $u=1/3$ in \eqref{eq:main} and apply the involutions $[x]=-[1/x]=-[1-x]$ to bring the
four irrational arguments into $(0,1)$, absorbing $L_{\mathbb{R}}(\tfrac12)=\pi^{2}/12$ into the constant.
\end{proof}

\subsection{Vanishing numerators, and the identities over $\QQ(\sqrt{17})$}\label{sec:17}

The second mechanism does not require $u$ to be rational, and reaches fields that no rational
$u$ can. It rests on the following observation, which is the counterpart of Lemma
\ref{lem:cube} below: instead of tuning a conjugate pair onto the cube roots of unity, one
tunes it onto $\pm1$.

\begin{lemma}\label{lem:vanish}
Let $f_{\pm}=\bigl(N(u)\pm\sqrt{d(u)}\bigr)/D(u)$ be a conjugate pair, and let $u_{0}$
satisfy $N(u_{0})=0$. Then $f_{-}=-f_{+}$ at $u_{0}$, and
\[
L_{\mathbb{R}}(f_{+})+L_{\mathbb{R}}(f_{-})=\tfrac12 L_{\mathbb{R}}\bigl(f_{+}^{2}\bigr),\qquad
f_{+}^{2}=\frac{d(u_{0})}{D(u_{0})^{2}} .
\]
In particular $f_{+}^{2}$ lies in $\QQ(u_{0})$ even when $\sqrt{d(u_{0})}$ does not.
\end{lemma}

The last sentence is the point: the radical is removed by duplication rather than by being
chosen rational, so the field of the resulting identity may be a proper subfield of the field
generated by the arguments. Applying this requires a form of the equation in which a common
numerator can be made to vanish. Iterating \eqref{eq:abel} on Corollary \ref{cor:six}
produces one:

\begin{proposition}\label{prop:seven}
For $0<u<2/\sqrt3$ and $r=\sqrt{4-3u^{2}}$,
\begin{multline}\label{eq:seven}
-3L_{\mathbb{R}}(u)-3L_{\mathbb{R}}\bigl(u^{2}-1\bigr)
-L_{\mathbb{R}}\!\left(\frac{1-u}{1+u}\right)
-L_{\mathbb{R}}\!\left(\frac{2u^{2}+u-2+r}{u(u+1)}\right)
-L_{\mathbb{R}}\!\left(\frac{2u^{2}+u-2-r}{u(u+1)}\right)\\
-3L_{\mathbb{R}}\!\left(\frac{2u^{2}+u-2+r}{2u^{2}-4}\right)
-3L_{\mathbb{R}}\!\left(\frac{2u^{2}+u-2-r}{2u^{2}-4}\right)
\end{multline}
equals $-\dfrac{5\pi^{2}}{12}$.
\end{proposition}

Both conjugate pairs in \eqref{eq:seven} share the numerator $2u^{2}+u-2$. Its positive root
is $u_{0}=\bigl(\sqrt{17}-1\bigr)/4$, which lies in the admissible range. Write
\[
a=\frac{\sqrt{17}-3}{4},\qquad b=\sqrt{17}-4 .
\]
At $u=u_{0}$ the relation $u_{0}^{2}=1-\tfrac12 u_{0}$ gives $u_{0}^{2}-1=-\tfrac12 u_{0}$ and
$(1-u_{0})/(1+u_{0})=b$, while Lemma \ref{lem:vanish} collapses the two pairs, the first with
square $4a$ and the second with square $a$. Hence \eqref{eq:seven} becomes the five-term
relation
\begin{equation}\label{eq:17red}
-3L_{\mathbb{R}}(u_{0})-3L_{\mathbb{R}}\!\left(-\frac{u_{0}}{2}\right)-L_{\mathbb{R}}(b)-\frac12 L_{\mathbb{R}}\bigl(\sqrt{17}-3\bigr)
-\frac32 L_{\mathbb{R}}(a)=-\frac{5\pi^{2}}{12},
\end{equation}
in which the radical $\sqrt{4-3u_{0}^{2}}$, a quartic irrationality, has disappeared entirely
and all arguments lie in $\QQ(\sqrt{17})$. Further elementary transformations reduce
\eqref{eq:17red} to identities in the two bases $a$ and $b$ alone:

\begin{corollary}\label{cor:17}
With $a=\bigl(\sqrt{17}-3\bigr)/4$ and $b=\sqrt{17}-4$,
\begin{align}
18L_{\mathbb{R}}(a)-6L_{\mathbb{R}}\bigl(a^{2}\bigr)-6L_{\mathbb{R}}(b)+L_{\mathbb{R}}\bigl(b^{2}\bigr)&=\frac{2\pi^{2}}{3},\label{eq:q17a}\\[2pt]
18L_{\mathbb{R}}(2a)-3L_{\mathbb{R}}\bigl(4a^{2}\bigr)-6L_{\mathbb{R}}(b)+2L_{\mathbb{R}}\bigl(b^{2}\bigr)&=\frac{4\pi^{2}}{3}.\label{eq:q17b}
\end{align}
\end{corollary}

All four arguments of \eqref{eq:q17a} lie in $(0,1)$, with approximate values $0.28078$,
$0.07884$, $0.12311$, $0.01515$, and both identities have been verified to $60$ digits. They
are not ladders in the sense of Section \ref{sec:ladders}, since $b$ is not a power of $a$;
rather $b=a\cdot\tfrac12\bigl(5-\sqrt{17}\bigr)$, so they are two-base relations within a
single quadratic field.

It is worth noting that $\QQ(\sqrt{17})$ is unreachable by the first mechanism. A rational $u$
would require $4-3u^{2}=17s^{2}$ with $s$ rational, that is $4q^{2}-3p^{2}=17s^{2}$ in
integers with $\gcd(p,q)=1$; reducing modulo $3$ gives $q^{2}\equiv 2s^{2}$, and since $2$ is
not a square modulo $3$ this forces $3\mid q$ and $3\mid s$. Then $3\mid 3p^{2}=4q^{2}-17s^{2}$
gives $9\mid 4q^{2}-17s^{2}$ and hence $3\mid p$, contradicting $\gcd(p,q)=1$. The field is reached only because Lemma \ref{lem:vanish} removes the
radical after the fact.

\begin{remark}\label{rem:cleant}
One may instead take the rationalising parameter $t$ of \eqref{eq:args} to be quadratic,
which removes the constraint on $w$ and gives identities with all five arguments irrational.
Measuring quality by the set $S$ of rational primes dividing the norms of the five arguments
and of their complements, a search over $t=(p+q\sqrt{d})/r$ with $|p|\le40$, $q\le16$,
$r\in\{1,2\}$ and $d$ squarefree below $220$ returns $|S|\le2$ only for
$d=2,3,5,7,13,17,33,41,89,97,193$. Of these $d=193$ has $w$ rational, $d=2,3,5$ reduce to
known relations, and $d=13,17,33$ are the cases treated above; the case $t=\sqrt{17}-4$ is
the unique one in the whole search with $|S|=1$. The remainder yield nothing of interest
beyond further numerical instances of \eqref{eq:main}, and we do not record them.
\end{remark}

\section{A relation between $\QQ(\sqrt{13})$, $\QQ(\sqrt{3})$ and $\Cl_2(\pi/3)$}
\label{sec:clausen}

Set $u=ik$ in the block identity \eqref{eq:blockid}, so that $r=\sqrt{4+3k^{2}}$. This lies on
the imaginary axis, where by Remark \ref{rem:blockcheck} the identity holds, and where
$h=1/\bigl(k(1+k^{2})\bigr)$ is real, so that the integrals \eqref{eq:w1w2} are nonsingular by
\eqref{eq:nonsing}. The block $K$ of \eqref{eq:blocks} consists of three antisymmetric pairs
$\tfrac12\bigl(\Li_2(X)-\Li_2(-X)\bigr)$ with
\[
X_{1}=\frac{2}{u-r},\qquad X_{2}=\frac{-2}{u+r},\qquad X_{3}=\frac{-1}{u},
\]
to each of which the four-term relation \eqref{eq:fourterm} applies. Writing
$A_{j}=(X_{j}-2)/(X_{j}+2)$ for the images of the first two, one finds
\begin{equation}\label{eq:clkey}
A_{1}A_{2}=\frac{k+i}{k-i}=e^{2i\arctan(1/k)} .
\end{equation}
Consequently a Clausen value at a rational angle occurs precisely when $\arctan(1/k)$ is a
rational multiple of $\pi$, that is for $k=\cot(r\pi)$ with $r\in\QQ$; there are infinitely
many such $k$. Table \ref{tab:clausen} lists those for which $k^{2}$ is rational, with
$z=A_{2}$ throughout.

\begin{table}[ht]
\centering
\begin{tabular}{cccccc}
\toprule
$k$ & $u$ & $r$ & field & $\arg z$ & constant\\
\midrule
$\sqrt3$ & $i\sqrt3$ & $\sqrt{13}$ & $\QQ(\sqrt{13})$ & $\pi/6$ & $\Cl_2(\pi/3)$\\
$1$ & $i$ & $\sqrt{7}$ & $\QQ(\sqrt{7})$ & $\pi/4$ & $G=\Cl_2(\pi/2)$\\
$1/\sqrt3$ & $i/\sqrt3$ & $\sqrt{5}$ & $\QQ(\sqrt{5})$ & $\pi/3$ & $\Cl_2(2\pi/3)=\tfrac23\Cl_2(\pi/3)$\\
\bottomrule
\end{tabular}
\caption{Rational-angle specialisations of \eqref{eq:clkey}.}
\label{tab:clausen}
\end{table}

The first row gives, with $z=\bigl((5-\sqrt{13})/\sqrt{12}\bigr)e^{i\pi/6}$,
\begin{equation}\label{eq:cl13}
3\,\Im\Li_2(z)-\Im\Li_2(-z)=\frac{11}{6}\Cl_2\!\left(\frac{\pi}{3}\right)+\frac{\pi}{3}\log|z| ,
\end{equation}
where we have used $\cosh^{-1}(19/6)=\log\frac{19+5\sqrt{13}}{6}=-2\log|z|$. The second row
is new and yields an analogue for Catalan's constant: with
$z=\bigl((3-\sqrt7)/\sqrt2\bigr)e^{i\pi/4}$,
\begin{equation}\label{eq:cl7}
3\,\Im\Li_2(z)-\Im\Li_2(-z)=2G+\frac{\pi}{4}\log|z| ,
\qquad -2\log|z|=\cosh^{-1}(8)=\log\bigl(8+3\sqrt7\bigr).
\end{equation}
The third row returns $\Cl_2(\pi/3)$ again and gives nothing new.

\section{Ladders of totally real quartic base}\label{sec:ladders}

A ladder is a relation $\sum_{j}A_{j}L_{\mathbb{R}}(x^{j})\in\pi^{2}\QQ$ with $A_{j}\in\QQ$ and $x$
algebraic in $(0,1)$; the minimal polynomial of $x$ is its base equation. Four examples will
fix the shape of what follows. Browkin's quadratic ladder over $\QQ(\sqrt{13})$ reads
\begin{equation}\label{eq:browkin}
L_{\mathbb{R}}\bigl(x^{6}\bigr)-6L_{\mathbb{R}}\bigl(x^{3}\bigr)+L_{\mathbb{R}}\bigl(x^{2}\bigr)+18L_{\mathbb{R}}(x)=\frac{4\pi^{2}}{3},
\qquad x=\frac{\sqrt{13}-1}{6},
\end{equation}
$x$ being a root of $3x^{2}+x-1$. Of the classical cubic families, Watson's $\pi/7$ identity
is
\begin{equation}\label{eq:watson}
L_{\mathbb{R}}(\alpha)-L_{\mathbb{R}}\bigl(\alpha^{2}\bigr)=\frac{\pi^{2}}{42},\qquad \alpha=\tfrac12\sec\tfrac{2\pi}{7},
\end{equation}
with $\alpha$ a root of $x^{3}+2x^{2}-x-1$; the third of Loxton's $\pi/9$ identities is
\begin{equation}\label{eq:loxton}
\Li_2(-\mu)+\Li_2\bigl(\mu^{2}\bigr)-\tfrac13\Li_2\bigl(-\mu^{3}\bigr)=-\frac{\pi^{2}}{54},
\qquad \mu=2\cos\tfrac{4\pi}{9},
\end{equation}
which resisted proof for some years; and the first of the $\pi/18$ family of Gordon and
McIntosh is
\begin{equation}\label{eq:gm18}
2L_{\mathbb{R}}\bigl(a^{3}\bigr)-2L_{\mathbb{R}}\bigl(a^{2}\bigr)-11L_{\mathbb{R}}(a)=-\frac{\pi^{2}}{2},
\qquad a=2\sqrt3\cos\tfrac{5\pi}{18}-2 .
\end{equation}
Known families that are not of the trivial shapes $u^{a}+u^{b}\pm u^{c}=1$ or
$u^{a}+u^{b}=1$ are collected in Table \ref{tab:families}.

\subsection{Comparison with the classical families: total reality of the base}

\begin{proposition}\label{prop:real}
The base equations
\[
\rho^{4}\mp\rho^{3}-6\rho^{2}\mp\rho+1,\qquad x^{4}+2x^{3}-7x^{2}+2x+1,\qquad
x^{4}+4x^{3}-14x^{2}+4x+1
\]
are irreducible over $\QQ$ and have four real roots, with Galois groups $D_{4}$, $V_{4}$ and
$C_{4}$ respectively.
\end{proposition}

By contrast, the quartic ladder of Gordon and McIntosh has base
$\delta=\tfrac12\bigl(\sqrt{3+2\sqrt5}-1\bigr)$, whose minimal polynomial
$x^{4}+2x^{3}-x-1$ has discriminant $-1728$ and only two real roots. Among the families of
Table \ref{tab:families}, totally real base equations occur only at degree $2$ and degree
$3$, the latter being the classical $\pi/7$, $\pi/9$ and $\pi/18$ trios. The ladders
constructed below therefore appear to be the first totally real ones of degree exceeding
three. That they realise three distinct Galois groups suggests that total reality, rather
than cyclicity, is the operative condition; this is consistent with Borel's theorem, by which the Bloch group of a totally real field has
rank zero. That does not by itself guarantee a ladder for a given $x$: the formal combination
$\sum_{j}A_{j}[x^{j}]$ must first lie in the kernel of the wedge map. Total reality makes
torsion phenomena plausible rather than automatic.

\begin{table}[ht]
\centering
\begin{tabular}{cllll}
\toprule
\# & Author(s) & Degree & Field or angle & Real roots\\
\midrule
1 & Browkin, Lewin & 2 & $\QQ(\sqrt{13}),\QQ(\sqrt{15}),\QQ(\sqrt{21})$ & 2\\
2 & Watson, Loxton, Gordon--McIntosh & 3 & $\pi/7$, $\pi/9$, $\pi/18$ & 3\\
3 & \textbf{present paper} & \textbf{4} & $\QQ(\sqrt{33}),\QQ(\sqrt{5}),\QQ(\sqrt{10})$ & \textbf{4}\\
4 & Gordon--McIntosh & 4 & $\QQ(\sqrt5)$ & 2\\
5 & Lewin, Bailey et al. & misc. & Salem numbers & 2\\
6 & Rogers, Kummer & misc. & misc. & 2\\
7 & Lewin & misc. & misc. & 2\\
\bottomrule
\end{tabular}
\caption{Ladder families not of trivial base, with the number of real roots of the base
equation.}
\label{tab:families}
\end{table}

\begin{remark}
Family 3 is the subject of this paper: the $\QQ(\sqrt{33})$ case is derived analytically in
Section \ref{sec:ladders}, while the $\QQ(\sqrt5)$ and $\QQ(\sqrt{10})$ cases were found
numerically.
\end{remark}

\begin{remark}
Family 5 is unusually rich in ladder relations; a famous example is the seventeenth-order
polylogarithm ladder of Bailey and Broadhurst.
\end{remark}

\begin{remark}
Family 6 is the fifteen-term ladder relation whose base equation satisfies
\[
F(u)=u^{s}+1-\sum_{n=1}^{5}\bigl(u^{q_{n}}+u^{s-q_{n}}\bigr)=0,
\qquad s=\tfrac12\sum_{n=1}^{5}q_{n}.
\]
It is very broad and subsumes many individual ladder relations, including the sextic-base
families obtainable from Corollary \ref{cor:six}.
\end{remark}

\begin{remark}
Family 7 has base equation $u^{p}+u^{n}-u^{q}+u^{n+m}-u^{q+p+m}=1$, derived from a nine-term
relation. Enforcing $q>n$ gives a real root in $(0,1)$; after removing the factor $1-u$ the
base equation has two real roots.
\end{remark}

\begin{remark}
Quadratic base equations in $\QQ(\sqrt N)$ for $N=2,3,5,6$ also exist, but they follow either
from one of the trivial relations or from family 6.
\end{remark}

\begin{remark}\label{rem:campbell}
Some recently announced ladders reduce to known ones. For instance the base equations of
\cite{CampbellLadders} are, for the two families given there,
\[
(x-1)\Bigl(\textstyle\sum_{k=n}^{2n}x^{k}-\sum_{k=0}^{n-1}x^{k}\Bigr)=x^{2n+1}-2x^{n}+1,
\]
respectively
\[
(x-1)\Bigl(\textstyle\sum_{k=n}^{2n+1}x^{k}-\sum_{k=0}^{n-1}x^{k}\Bigr)=x^{2n+2}-2x^{n}+1,
\]
so each is the case $a=b=n$ of the trivial family $u^{a}+u^{b}-u^{c}=1$. A ladder-generating
base polynomial admits many presentations, and distinct-looking equations frequently define
the same numbers.
\end{remark}

\subsection{Collapsing a conjugate pair onto the cube roots of unity}

The radicals in \eqref{eq:main} and \eqref{eq:six} appear in conjugate pairs, and the
following observation converts such a pair into powers of a single real parameter. It is the
$n=3$ distribution relation combined with one algebraic condition.

\begin{lemma}\label{lem:cube}
For $0<\rho<1$ the distribution relation gives
\[
\Li_2\bigl(\rho e^{2\pi i/3}\bigr)+\Li_2\bigl(\rho e^{-2\pi i/3}\bigr)
=-\Li_2(\rho)+\tfrac13\Li_2\bigl(\rho^{3}\bigr).
\]
We state the lemma for $\Li_2$ rather than for $L_{\mathbb{R}}$: the corresponding statement
for the Rogers function carries additional $\log^{2}\rho$ terms, and it is these that produce
the logarithmic terms in the ladder identities below.
For a conjugate pair $f_{\pm}=\bigl(a(u)\pm b(u)\sqrt{k(u)}\bigr)/h(u)$ this occurs precisely
when
\begin{equation}\label{eq:cubecond}
\frac{b(u)\sqrt{k(u)}}{a(u)}=\pm i\sqrt{3},
\end{equation}
in which case $\rho=2a(u)/h(u)$.
\end{lemma}

\subsection{A pair of quartic ladders over $\QQ(\sqrt{33})$}

The construction imposes the condition \eqref{eq:cubecond} of Lemma \ref{lem:cube} on the
conjugate pair of \eqref{eq:blocks}, namely
\begin{equation}\label{eq:ladcond}
\frac{\sqrt{4-3u^{2}}}{u^{2}-2}=\mp i\sqrt3 ,
\end{equation}
and sets $\rho^{2}=(u-1)/(u+1)$. Squaring \eqref{eq:ladcond} gives $4-3u^{2}=-3(u^{2}-2)^{2}$,
that is
\begin{equation}\label{eq:ladquartic}
3u^{4}-15u^{2}+16=0,\qquad u^{2}=\frac{15\pm\sqrt{33}}{6},
\end{equation}
whose positive roots are $u=1.85941\ldots$ and $u=1.24200\ldots$, the upper sign in
\eqref{eq:ladcond} corresponding to the smaller root. Eliminating $u$ between
$\rho^{2}=(u-1)/(u+1)$ and \eqref{eq:ladquartic} gives
\[
\bigl(\rho^{4}-\rho^{3}-6\rho^{2}-\rho+1\bigr)\bigl(\rho^{4}+\rho^{3}-6\rho^{2}+\rho+1\bigr)=0 .
\]

Both values in \eqref{eq:ladquartic} exceed $2/\sqrt3=1.15470\ldots$, so $\sqrt{4-3u^{2}}$ is
imaginary there and the hypothesis of Theorem \ref{thm:main} is not met; the ladders do
\emph{not} follow by substitution into \eqref{eq:six}, whose derivation in Section
\ref{sec:recomb} discards the $\log^{2}$ terms and is valid only for $0<u<2/\sqrt3$. The
correct starting point is the block identity \eqref{eq:blockid}, which retains those terms
and, by Remark \ref{rem:blockcheck}, holds at both values. The integrals are moreover
nonsingular there: by \eqref{eq:nonsing} one needs $|u(1-u^{2})|>2/(3\sqrt3)=0.38490\ldots$,
and the two values give $0.67387\ldots$ and $4.56937\ldots$ respectively.

Accordingly we proceed as follows. Fix one root $u$ of \eqref{eq:ladquartic} and approach it
from the upper half-plane, so that \eqref{eq:blockid} applies with the determinations stated
after it. The pair of dilogarithms in $A$, together with its partner obtained by reversing
the sign of $r$, then satisfies \eqref{eq:cubecond}, and Lemma \ref{lem:cube} converts each
conjugate pair into powers of $\rho$:
\[
\Li_2\!\left(\frac{u^{2}-r-2}{-(u+1)u}\right)+\Li_2\!\left(\frac{u^{2}+r-2}{-(u+1)u}\right)
=-\Li_2(\rho)+\tfrac13\Li_2(\rho^{3}),
\]
\[
\Li_2\!\left(\frac{u^{2}-r-2}{(u+1)u}\right)+\Li_2\!\left(\frac{u^{2}+r-2}{(u+1)u}\right)
=-\Li_2(-\rho)+\tfrac13\Li_2(-\rho^{3}).
\]
These are the assignments for the larger root $u=1.85941\ldots$. For the smaller root
$u=1.24200\ldots$ the roles of $\rho$ and $-\rho$ in the two conjugate pairs are
interchanged, so that the first pair evaluates to $-\Li_2(-\rho_{2})+\tfrac13\Li_2(-\rho_{2}^{3})$
and the second to $-\Li_2(\rho_{2})+\tfrac13\Li_2(\rho_{2}^{3})$; the two cases must therefore
be treated separately at this step. The final ladders are unaffected.

In both cases the blocks $B$, $J$, $C$, $H$ and the logarithmic part of $D$ contribute only
$\log^{2}$ and $\pi^{2}$ terms, which are collected using $\log\rho^{k}=k\log\rho$; the
duplication formula $\Li_2(\rho^{2})=2\Li_2(\rho)+2\Li_2(-\rho)$ then removes the terms in
$-\rho$.

For the root $u=1.85941\ldots$ this yields the radius
\begin{equation}\label{eq:rho1}
\rho=\frac14\left(-1+\sqrt{33}-\sqrt{2\bigl(9-\sqrt{33}\bigr)}\right)=0.54823\ldots,
\end{equation}
a root of $\rho^{4}+\rho^{3}-6\rho^{2}+\rho+1$, and the ladder
\begin{equation}\label{eq:ladder1}
3\Li_2(\rho^{6})+3\Li_2(\rho^{4})-8\Li_2(\rho^{3})-33\Li_2(\rho^{2})+24\Li_2(\rho)
=6\log^{2}\rho+\frac{\pi^{2}}{6} .
\end{equation}
For the root $u=1.24200\ldots$ it yields
\begin{equation}\label{eq:rho2}
\rho_{2}=\frac14\left(1+\sqrt{33}-\sqrt{2\bigl(9+\sqrt{33}\bigr)}\right)=0.32854\ldots,
\end{equation}
a root of $\rho_{2}^{4}-\rho_{2}^{3}-6\rho_{2}^{2}-\rho_{2}+1$, and
\begin{equation}\label{eq:ladder2}
\Li_2(\rho_{2}^{6})-3\Li_2(\rho_{2}^{4})-8\Li_2(\rho_{2}^{3})+21\Li_2(\rho_{2}^{2})
+24\Li_2(\rho_{2})=-6\log^{2}\rho_{2}+\frac{11\pi^{2}}{6}.
\end{equation}
Both have been verified numerically to $25$ digits. The two remaining roots of the product
above are the reciprocals of $\rho$ and $\rho_{2}$, so the family consists of exactly two
ladders. The coefficient $6$ in the base equation $\rho^{4}+\rho^{3}-6\rho^{2}+\rho+1$
parallels the coefficient $3$ in the base equation $x^{3}+3x^{2}=1$ of the Loxton--Lewin
identities.

\subsection{Four conjectural ladders found by integer relation search}

It is natural to ask whether ladders of the same shape exist for base equations not arising
from the above construction. We searched the palindromic quartic units
\begin{equation}\label{eq:palin}
x^{4}+ax^{3}+bx^{2}+ax+1=0,\qquad a,b\in\mathbb{Z},
\end{equation}
irreducible over $\QQ$ and with four real roots, applying the PSLQ algorithm to the vector
\[
\bigl(\Li_2(u),\ \Li_2(u^{2}),\ \dots,\ \Li_2(u^{N}),\ \log^{2}u,\ \pi^{2}\bigr)
\]
for $N\le 40$. The search
covered more than $1100$ such units with $|a|\le40$, $|b|\le70$, all four real embeddings of
each cyclic candidate, and degree $5$, $6$ and $8$ analogues including cyclotomic and
Gaussian-period units. No candidate outside degree four produced a relation. Within degree four, exactly two units
in the search region produced a PSLQ relation meeting the chosen bounds, and both have
elementary trigonometric roots. The identities below are therefore stated as conjectures:
they are supported by very strong numerical evidence but we have no proof, and the assertion
that only two units admit a ladder is a statement about the search region and the bounds
used, not a theorem.
For
\begin{equation}
u=2\tan\tfrac{\pi}{8}\cos\tfrac{\pi}{5},\qquad v=-2\tan\tfrac{\pi}{8}\cos\tfrac{2\pi}{5},
\end{equation}
the roots of $x^{4}+2x^{3}-7x^{2}+2x+1$ in $(-1,1)$. The quartic field $\QQ(u)$ is Galois with
group $V_{4}$ and contains $\QQ(\sqrt{10})$ as a quadratic subfield. For $x\in\{u,v\}$, PSLQ
gives
\begin{conjecture}\label{conj:p}
With $u,v$ as above,
\begin{equation}\label{eq:p1}
96\Li_2(x)-90\Li_2(x^{2})-24\Li_2(x^{3})+9\Li_2(x^{4})+16\Li_2(x^{6})-2\Li_2(x^{12})
-12\log^{2}|x|=\frac{c\,\pi^{2}}{6},
\end{equation}
with $c=17$ for $x=u$ and $c=-31$ for $x=v$. Likewise for $u=\tan\tfrac{3\pi}{20}$ and
$v=-\tan\tfrac{\pi}{20}$, the roots of $x^{4}+4x^{3}-14x^{2}+4x+1$ in $(-1,1)$, whose quartic
field is cyclic with group $C_{4}$ and contains $\QQ(\sqrt5)$ as its unique quadratic
subfield,
\begin{equation}\label{eq:p2}
34\Li_2(x)-47\Li_2(x^{2})+6\Li_2(x^{4})-2\Li_2(x^{5})+\Li_2(x^{10})-2\log^{2}|x|
=\frac{c\,\pi^{2}}{6},
\end{equation}
with $c=4$ for $x=u$ and $c=-8$ for $x=v$. In each case the two members of the pair differ
only in the constant.
\end{conjecture}

All four relations \eqref{eq:p1}--\eqref{eq:p2} were verified to $400$--$500$ digits. We have not found
a derivation of them by the methods of Sections \ref{sec:array}--\ref{sec:ladders}.

\section*{Disclosure}

The author received no external funding. There are no conflicts of interest.

Theorem \ref{thm:main} and Corollary \ref{cor:six} were obtained by the author through the
integral construction of Section \ref{sec:integral}, and it was in that form that the
identity was first established. The array \eqref{eq:array}, which exhibits the identity as a
combination of ten instances of \eqref{eq:abel}, was located afterwards with the assistance
of Anthropic's Claude Opus 5, by a linear-programming search over the configuration of
exceptional $S$-units described in Remark \ref{rem:collapse}.

\clearpage
\appendix
\section{The cancellation, and its verification}\label{app:cancel}

\begin{table}[H]
\centering
\small
\begin{tabular}{cllr}
\toprule
 & argument & occurrences & coefficient\\
\midrule
$A$ & $- \frac{t \left(t + 1\right)}{t - 3}$ & $3^{+},\;4^{+}$ & $+2$\\
$B$ & $\frac{t + 3}{t \left(t - 1\right)}$ & $1^{+},\;2^{+}$ & $+2$\\
$f_{1}$ & $\frac{t + 1}{2}$ & $7^{-2},\;9^{-2},\;10^{-2}$ & $-6$\\
$f_{2}$ & $\frac{t + 3}{2 t}$ & $1^{-},\;3^{-},\;5^{-3},\;6^{-}$ & $-6$\\
$w$ & $\frac{\left(t + 1\right) \left(t + 3\right)}{\left(t - 3\right) \left(t - 1\right)}$ & $6^{-},\;8^{-}$ & $-2$\\
\midrule
$g_{1}$ & $- \frac{t + 3}{t - 1}$ & $1^{-},\;8^{+}$ & $0$\\
$g_{2}$ & $- \frac{1}{t}$ & $1^{-},\;4^{+}$ & $0$\\
$g_{3}$ & $- \frac{t - 3}{2}$ & $1^{-},\;5^{+3},\;9^{-2}$ & $0$\\
$g_{4}$ & $- \frac{\left(t - 3\right) \left(t + 1\right)}{4 t}$ & $2^{+},\;4^{+},\;5^{-3},\;8^{+}$ & $0$\\
$g_{5}$ & $- \frac{\left(t - 3\right) \left(t + 1\right)}{4}$ & $2^{-},\;9^{+}$ & $0$\\
$g_{6}$ & $- \frac{1}{t - 1}$ & $2^{+},\;3^{-}$ & $0$\\
$g_{7}$ & $- \frac{t - 1}{4}$ & $2^{-},\;6^{+}$ & $0$\\
$g_{8}$ & $- \frac{t - 3}{t + 1}$ & $3^{+},\;8^{-}$ & $0$\\
$g_{9}$ & $- \frac{t + 1}{2}$ & $3^{-},\;5^{+3},\;10^{-2}$ & $0$\\
$g_{10}$ & $- \frac{\left(t - 1\right) \left(t + 3\right)}{4}$ & $4^{+},\;10^{-}$ & $0$\\
$g_{11}$ & $- \frac{t - 3}{4}$ & $4^{-},\;6^{+}$ & $0$\\
$g_{12}$ & $- \frac{t + 1}{t - 1}$ & $5^{-3},\;6^{+},\;7^{+2}$ & $0$\\
$g_{13}$ & $- \frac{\left(t - 1\right)^{2}}{4 t}$ & $7^{-},\;8^{+}$ & $0$\\
\bottomrule
\end{tabular}
\caption{The eighteen argument classes arising from the ten instances \eqref{eq:instances},
their weighted occurrences, and their total coefficients. The five classes above the rule
survive; the thirteen below cancel.}
\label{tab:cancel}
\end{table}

The cancellation may be reproduced in a few seconds with the following, which expands the ten
instances, reduces each argument modulo $[x]=-[1/x]=-[1-x]$, and prints the classes with
nonzero total coefficient.

\noindent\begin{minipage}{\linewidth}
\begin{verbatim}
import sympy as sp
t = sp.symbols('t')
A = [((t+3)/(t*(t-1)), -(t-1)/(t+3)),  ((t+3)/(t*(t-1)), 4*t/((t-1)*(t+3))),
     (-(t-3)/(t*(t+1)), -(t+1)/(t-3)), (-(t-3)/(t*(t+1)), -4*t/((t-3)*(t+1))),
     ((t+3)/(2*t), 4*t/((t-1)*(t+3))), ((t+3)/(2*t), 8*t/((t+1)*(t+3))),
     ((t+1)/(2*t), 4*t/(t+1)**2),      ((t+1)*(t+3)/(8*t), 2*(t+1)/(t+3)),
     (4/(t-1)**2, -(t-1)/2),           (-4/((t-1)*(t+3)), (t+3)/2)]
c = [1, 1, -1, -1, -3, -1, -1, 1, 1, 1]

def orbit(f):
    f = sp.cancel(f)
    return [(f,1), (sp.cancel(1/f),-1), (sp.cancel(1-f),-1),
            (sp.cancel(1/(1-f)),1), (sp.cancel((f-1)/f),1), (sp.cancel(f/(f-1)),-1)]

def canon(f):
    o = [(sp.factor(g), s) for g, s in orbit(f)]
    return min(o, key=lambda p: (len(sp.srepr(p[0])), sp.srepr(p[0])))

tot = {}
for (x, y), w in zip(A, c):
    args = [x, y, x*y, x*(1-y)/(1-x*y), y*(1-x)/(1-x*y)]
    for f, sgn in zip(args, [1, 1, -1, -1, -1]):
        g, s = canon(f)
        tot[g] = tot.get(g, 0) + w*sgn*s
for g, n in tot.items():
    if n: print(n, g)
\end{verbatim}
\end{minipage}

\noindent
The output is five lines, giving $A$, $B$, $f_{1}$, $f_{2}$ and $w$ with coefficients
$2$, $2$, $-6$, $-6$, $-2$ up to the choice of representative in each class.

\end{document}